\documentclass[3p,times]{elsarticle}

\journal{Elsevier}

\usepackage{amsmath}
\usepackage{mathrsfs} 
\newcommand{\domD}{\mathscr{D}}
\renewcommand{\pi}{\piup}
\renewcommand{\Re}{\operatorname{Re}}
\renewcommand{\Im}{\operatorname{Im}}
\DeclareMathOperator{\diag}{diag}
\DeclareMathOperator{\arsinh}{arsinh}
\DeclareMathOperator{\Si}{Si}
\DeclareMathOperator{\OO}{O}
\DeclareMathOperator{\E}{e}
\DeclareMathOperator{\I}{i}
\newcommand{\D}{\,\mathrm{d}}
\newcommand{\mathbd}{\boldsymbol}

\newdefinition{definition}{Definition}[section]
\newtheorem{theorem}{Theorem}[section]
\newtheorem{lemma}[theorem]{Lemma}
\newtheorem{proposition}[theorem]{Proposition}
\newtheorem{remark}{Remark}[section]
\newproof{proof}{Proof}
\numberwithin{equation}{section}

\begin{document}

\begin{frontmatter}

\title{Error analyses of Sinc-collocation methods for exponential decay initial
value problems\tnoteref{mytitlenote}}
\tnotetext[mytitlenote]{This work was partially supported by JSPS
Grant-in-Aid for Scientific Research (C) JP23K03218.}

\author{Tomoaki Okayama}
\address{Graduate School of Information Sciences, Hiroshima City University,
3-4-1, Ozuka-higashi, Asaminami-ku, Hiroshima 731-3194, Japan}
\ead{okayama@hiroshima-cu.ac.jp}

\author{Ryota Hara}
\address{PC Assist Co., Ltd.,
Gojinsha hiroshima kamiyamachi building 5F,
11-10, Motomachi, Naka-ku, Hiroshima 730-0011, Japan}

\author{Shun'ichi Goto}
\address{IZUMI Co., Ltd.
3-3-1 Futabanosato, Higashi-ku, Hiroshima 732-0057, Japan}




\begin{abstract}
Nurmuhammad et al.\ developed the
Sinc-Nystr\"{o}m methods for initial value problems
in which the solutions exhibit exponential decay end behavior.
In these methods, the Single-Exponential (SE) transformation
or the Double-Exponential (DE) transformation
is combined with the Sinc approximation.
Hara and Okayama improved on these transformations
to attain a better convergence rate,
which was later supported by theoretical error analyses.
However, these methods have a computational drawback
owing to the inclusion of a special function in the basis functions.
To address this issue, Okayama and Hara
proposed Sinc-collocation methods, which do not include any
special function in the basis functions.
This study conducts error analyses of these methods.
\end{abstract}


\end{frontmatter}

\section{Introduction and summary}

This study focuses on numerical solutions
for systems of initial value problems of the following form:
\begin{equation}
\begin{cases}
 \mathbd{y}'(t) = K(t)\mathbd{y}(t) + \mathbd{g}(t),\quad t\geq 0,\\
 \mathbd{y}(0) = \mathbd{r},
\end{cases}
\label{eq:ODE}
\end{equation}
where $K(t)$ is an $m\times m$ matrix whose $(i, j)$ elements are
$k_{ij}(t)$, and $\mathbd{y}(t)$, $\mathbd{g}(t)$, and $\mathbd{r}$ are
$m$-dimensional vectors. In this study, the solution $\mathbd{y}(t)$ is
assumed to decay exponentially as $t\to\infty$.
For such a case,
Nurmuhammad et al.~\cite{Nurmuhammad}
proposed the Sinc-Nystr\"{o}m methods
by means of the Sinc indefinite integration
and two types of variable transformations:
Single-Exponential (SE) transformation and
the Double-Exponential (DE) transformation.
In their numerical experiments,
these methods exhibited exponential convergence with respect to the number
of sampling points $l$,
which is much faster than polynomial convergence,
such as $\OO(l^{-4})$ (e.g.,\ classic Runge--Kutta or Gauss--Legendre method).
It should be noted that
such a fast convergence was also observed for a stiff problem.

The convergence was further improved by
Hara and Okayama~\cite{HaraOkayamaP}.
They improved the SE and DE transformations in the Sinc-Nystr\"{o}m methods,
and showed their superiority by numerical experiments.
They~\cite{HaraOkayama} later theoretically showed that
the Sinc-Nystr\"{o}m method combined with the SE transformation
(called SE-Sinc-Nystr\"{o}m method) can attain $\OO(\exp(-c\sqrt{l}))$,
and the Sinc-Nystr\"{o}m method combined with the DE transformation
(called DE-Sinc-Nystr\"{o}m method) can attain $\OO(\exp(-cl/\log l))$.
These convergence rates were derived
with the assumptions that $\|A_{lm}^{-1}\|_{\infty}$ and $\|B_{lm}^{-1}\|_{\infty}$ do not
diverge exponentially with respect to $l$,
where $A_{lm}$ and $B_{lm}$
denote the coefficient matrices of the system of linear equations
for the SE- and DE-Sinc-Nystr\"{o}m methods, respectively.
These assumptions seem reasonable, given their numerical observations.

However, the Sinc-Nystr\"{o}m methods have a disadvantage
in terms of computational cost.
This is because the basis functions of these methods include the sine integral,
which is a special function
defined by
\begin{equation}
 \Si(x) = \int_0^x\frac{\sin t}{t}\D t.
\label{eq:Si}
\end{equation}
To eliminate this disadvantage,
Okayama and Hara~\cite{OkayamaHara}
proposed Sinc-collocation methods so that
basis functions do not include any special function.
Their methods were derived by means of
the Sinc approximation with a boundary treatment,
combined with the SE/DE transformations.
We note that such methods were already derived for
initial value problems over a finite interval~\cite{okayama18:_theor_sinc},
but not derived for the present case,
i.e., initial value problems with exponential decay end behavior
over the semi-infinite interval $(0,\infty)$.
Based on the results of numerical experiments,
they reported that the Sinc-collocation methods
achieved the same precision as the Sinc-Nystr\"{o}m methods,
but with significantly lower computational costs.

The objective of this study is to provide a theoretical explanation
for their report.
The key to this project is the error analysis
of the Sinc approximation with a boundary treatment.
In the case where the SE transformation is used with the Sinc approximation,
the error analysis was provided~\cite{OkayamaHamada},
whereas that was not provided
in the case where the DE transformation is used.
Therefore, we present the error analysis of the latter case.
Then, with the aid of the error analyses of the Sinc-Nystr\"{o}m methods,
we analyze the error of the Sinc-collocation methods.
As a result, it is shown that
the convergence rate of the Sinc-collocation methods
is slightly worse than that of the Sinc-Nystr\"{o}m methods.
This is not a negative but positive result for the Sinc-collocation methods;
we can conclude that
without using any special function,
the Sinc-collocation methods can achieve almost the
same convergence rate as the Sinc-Nystr\"{o}m methods.

The remainder of this paper is organized as follows.
As a preliminary,
we describe the Sinc approximation and the Sinc indefinite integration
and their convergence theorems in Sect.~\ref{sec:preliminary}.
The Sinc-Nystr\"{o}m methods and their error analyses are described
in Sect.~\ref{sec:SNM}.
The Sinc-collocation methods and their error analyses,
the main results of this study, are described
in Sect.~\ref{sec:SCM}.
The proofs of the new theorems in Sect.~\ref{sec:preliminary}
are provided in Sect.~\ref{sec:proofs-DE-Sinc}.
The proofs of the new theorems in Sect.~\ref{sec:SCM}
are provided in Sect.~\ref{sec:proofs}.

\section{Sinc approximation and Sinc indefinite integration}
\label{sec:preliminary}

\subsection{Sinc approximation}

The Sinc approximation is a function approximation formula
over the real axis $\mathbb{R}$, which is expressed as
\begin{equation}
 F(x)\approx \sum_{j=-M}^N F(jh)S(j,h)(x),\quad x\in\mathbb{R},
\label{eq:Sinc-approximation}
\end{equation}
where $S(k,h)(x)$ is the so-called ``Sinc function'' defined by
\[
 S(j,h)(x) = \frac{\sin[\pi(x - jh)/h]}{\pi(x - jh)/h},
\]
and $h$, $M$, and $N$ are suitably selected depending on
a given positive integer $n$.

\subsubsection{SE-Sinc approximation and its convergence theorem}

To apply the Sinc approximation~\eqref{eq:Sinc-approximation},
$F$ should be defined on the entire real axis $\mathbb{R}$.
If the function to be approximated, denoted by $f(t)$, is
defined for $t\geq 0$,
we should employ a variable transformation that maps $\mathbb{R}$ onto
$(0,\infty)$.
Especially in the case where $f(t)$ decays exponentially as $t\to\infty$,
such as $f(t)=\sqrt{t}\E^{-t}$,
the following variable transformation
\[
 t = \psi(x) = \log(1 + \E^x)
\]
was proposed~\cite{OkaShinKatsu}.
We refer to this transformation as the Single-Exponential (SE) transformation.
With the SE transformation, putting $F(x)=f(\psi(x))$,
we can apply~\eqref{eq:Sinc-approximation} as
\[
 f(\psi(x)) \approx\sum_{j=-M}^N f(\psi(jh))S(j,h)(x),\quad x\in\mathbb{R},
\]
which is equivalent to
\begin{equation}
 f(t)\approx\sum_{j=-M}^N f(\psi(jh))S(j,h)(\psi^{-1}(t)),\quad t\in (0,\infty).
\label{eq:SE-Sinc-approximation}
\end{equation}
We refer to this approximation as the SE-Sinc approximation.

For efficient approximation through~\eqref{eq:Sinc-approximation},
$F$ should be analytic on a strip domain
\[
 \domD_d = \{\zeta\in\mathbb{C}: |\Im\zeta|< d\}
\]
for a positive constant $d$. Therefore,
for efficient approximation through~\eqref{eq:SE-Sinc-approximation},
$f$ should be analytic on a translated domain
\[
 \psi(\domD_d) = \{z=\psi(\zeta):\zeta\in\domD_d\}.
\]
Actually, the convergence of the SE-Sinc approximation was
analyzed as follows.

%

\begin{theorem}[Okayama et al.~{\cite[Theorem~2.2]{OkaShinKatsu}}]
\label{thm:SE-Sinc}
Assume that $f$ is analytic in $\psi(\domD_d)$ with $0<d<\pi$.
Furthermore, assume that
there exist positive constants $C_{\ddagger}$, $\alpha$, and $\beta$
such that
\begin{equation}
|f(z)|\leq C_{\ddagger} \left|\frac{z}{1+z}\right|^{\alpha} |\E^{-z}|^{\beta}
\label{eq:f-bound}
\end{equation}
holds for all $z\in\psi(\domD_d)$.
Let $\mu=\min\{\alpha,\beta\}$, let $M$ and $N$ be defined as
\begin{equation}
\left\{
\begin{aligned}
M&=n,& N&=\left\lceil\frac{\alpha}{\beta} n\right\rceil
 &&(\text{if}\,\,\,\mu = \alpha),\\
N&=n,& M&=\left\lceil\frac{\beta}{\alpha} n\right\rceil
 &&(\text{if}\,\,\,\mu = \beta),
\end{aligned}
\label{eq:def-MN}
\right.
\end{equation}
and let $h$ be defined as
\begin{equation}
h = \sqrt{\frac{\pi d}{\mu n}}.
\label{eq:def-h}
\end{equation}
Then, there exists a positive constant $C$ independent of $n$ such that
\[
 \sup_{t\in(0,\infty)}\left|
  f(t) - \sum_{j=-M}^N f(\psi(jh))S(j,h)(\psi^{-1}(t))
 \right|
\leq C \sqrt{n} \E^{-\sqrt{\pi d \mu n}}.
\]
\end{theorem}

\subsubsection{SE-Sinc approximation with a boundary treatment and its convergence theorem}

According to Theorem~\ref{thm:SE-Sinc},
$f$ should satisfy~\eqref{eq:f-bound},
which requires $f$ to be zero at the boundary of $(0,\infty)$.
Here, let $\tilde{f}$ be a function with general boundary values
\begin{equation}
\left\{
\begin{aligned}
\lim_{t\to\infty} \tilde{f}(t) &= p,\\
\lim_{t\to 0} \tilde{f}(t)     &= q.
\end{aligned}
\right.
\label{eq:def-p-and-q}
\end{equation}
Okayama and Hamada~\cite{OkayamaHamada} considered
the following function
\begin{equation}
 f(t) = \tilde{f}(t) - \frac{q + p(\E^t - 1)}{\E^t},
\label{eq:Hamada-func}
\end{equation}
which is zero at the boundary of $(0,\infty)$.
Then, they considered application of~\eqref{eq:SE-Sinc-approximation},
which is equivalent to
\begin{equation}
 \tilde{f}(t)\approx
\frac{q + p(\E^t - 1)}{\E^t}
+\sum_{j=-M}^N\left(\tilde{f}(\psi(jh)) - \frac{q + p\E^{jh}}{1 + \E^{jh}}\right)
S(j,h)(\psi^{-1}(t)),\quad t\in(0,\infty).
\label{eq:general-SE-Sinc}
\end{equation}
We refer to this approximation as the SE-Sinc approximation with a
boundary treatment.
Its convergence was analyzed as follows.

\begin{theorem}[Okayama and Hamada~{\cite[Theorem~3]{OkayamaHamada}}]
\label{thm:SE-Sinc-boundary}
For a given function $\tilde{f}$,
let $p$ and $q$ be defined by~\eqref{eq:def-p-and-q},
and let $f$ be defined by~\eqref{eq:Hamada-func}.
Assume that $f$ is analytic in $\psi(\domD_d)$ with $0<d<\pi$.
Furthermore, assume that
there exist positive constants $C_{\ddagger}$, $\alpha$, and $\beta$
such that~\eqref{eq:f-bound}
holds for all $z\in\psi(\domD_d)$.
Let $\mu=\min\{\alpha,\beta\}$, let $M$ and $N$ be defined as~\eqref{eq:def-MN},
and let $h$ be defined as~\eqref{eq:def-h}.
Then, there exists a positive constant $C$ independent of $n$ such that
\begin{align*}
\sup_{t\in(0,\infty)}\left|
\tilde{f}(t) - \left[\frac{q + p (\E^t - 1)}{\E^t}
+\sum_{j=-M}^N\left\{\tilde{f}(\psi(jh)) - \frac{q + p\E^{jh}}{1+\E^{jh}}\right\}
S(j,h)(\psi^{-1}(t))
\right]
\right|
\leq C \sqrt{n} \E^{-\sqrt{\pi d \mu n}}.
\end{align*}
\end{theorem}

\subsubsection{DE-Sinc approximation and its convergence theorem (new result)}

The SE transformation is not the only variable transformation that maps
$\mathbb{R}$ onto $(0,\infty)$.
In fact, another variable transformation
\[
 t = \phi(x) = \log(1 + \E^{\pi\sinh x})
\]
was proposed~\cite{OkayamaSinc}.
We refer to this transformation as the Double-Exponential (DE) transformation.
With the DE transformation, putting $F(x)=f(\phi(x))$,
we can apply~\eqref{eq:Sinc-approximation} as
\[
 f(\phi(x)) \approx\sum_{j=-n}^n f(\phi(jh))S(j,h)(x),\quad x\in\mathbb{R},
\]
where we choose $M=N=n$ in this case.
This is equivalent to
\begin{equation}
 f(t)\approx\sum_{j=-n}^n f(\phi(jh))S(j,h)(\phi^{-1}(t)),\quad t\in (0,\infty).
\label{eq:DE-Sinc-approximation}
\end{equation}
We refer to this approximation as the DE-Sinc approximation.
%

For efficient approximation through~\eqref{eq:DE-Sinc-approximation},
$f$ should be analytic on a translated domain
\[
 \phi(\domD_d) = \{z=\phi(\zeta):\zeta\in\domD_d\}.
\]
We provide its convergence theorem as follows.
The proof is given in Sect.~\ref{sec:proofs-DE-Sinc}.

\begin{theorem}
\label{thm:DE-Sinc}
Assume that $f$ is analytic in $\phi(\domD_d)$ with $0<d<\pi/2$.
Furthermore, assume that
there exist positive constants $C_{\ddagger}$ and $\mu$ with $\mu\leq 1$
such that
\begin{equation}
|f(z)|\leq C_{\ddagger} \left|z\right|^{\mu} |\E^{-z}|^{\mu}
\label{eq:f-bound-DE}
\end{equation}
holds for all $z\in\phi(\domD_d)$.
Let $h$ be defined as
\begin{equation}
h = \frac{\arsinh(dn/\mu)}{n}.
\label{eq:def-h-DE}
\end{equation}
Then, there exists a positive constant $C$ independent of $n$ such that
\[
 \sup_{t\in(0,\infty)}\left|
  f(t) - \sum_{j=-n}^n f(\phi(jh))S(j,h)(\phi^{-1}(t))
 \right|
\leq C \E^{-\pi d n/\arsinh(dn/\mu)}.
\]
\end{theorem}

\begin{remark}
The error of~\eqref{eq:DE-Sinc-approximation} was also
analyzed by Okayama~\cite[Theorem~2.9]{OkayamaSinc}.
However, in the existing theorem, the formula of $h$ is set as
\[
 h = \frac{\log(2 d n/\mu)}{n},
\]
which is different from~\eqref{eq:def-h-DE}.
In the present paper, we set $h$ as~\eqref{eq:def-h-DE},
and therefore we establish another theorem here.
\end{remark}

\subsubsection{DE-Sinc approximation with a boundary treatment and its convergence theorem (new result)}

By the condition~\eqref{eq:f-bound-DE},
Theorem~\ref{thm:DE-Sinc} also requires
$f$ to be zero at the boundary of $(0,\infty)$.
In the case of $\tilde{f}$ with general boundary values,
we consider the same function $f$ as~\eqref{eq:Hamada-func}.
Then, we apply~\eqref{eq:DE-Sinc-approximation},
which is equivalent to
\begin{equation}
 \tilde{f}(t)\approx
\frac{q + p(\E^t - 1)}{\E^t}
+\sum_{j=-n}^n\left(\tilde{f}(\phi(jh)) - \frac{q + p\E^{\pi\sinh(jh)}}{1 + \E^{\pi\sinh(jh)}}\right)
S(j,h)(\phi^{-1}(t)),\quad t\in(0,\infty).
\label{eq:general-DE-Sinc}
\end{equation}
We refer to this approximation as the DE-Sinc approximation with a
boundary treatment.
We provide its convergence theorem as follows.
The proof is given in Sect.~\ref{sec:proofs-DE-Sinc}.

\begin{theorem}
\label{thm:DE-Sinc-boundary}
For a given function $\tilde{f}$,
let $p$ and $q$ be defined by~\eqref{eq:def-p-and-q},
and let $f$ be defined by~\eqref{eq:Hamada-func}.
Assume that $f$ is analytic in $\phi(\domD_d)$ with $0<d<\pi/2$.
Furthermore, assume that
there exist positive constants $C_{\ddagger}$ and $\mu$ with $\mu\leq 1$
such that~\eqref{eq:f-bound-DE}
holds for all $z\in\phi(\domD_d)$.
Let $h$ be defined as~\eqref{eq:def-h-DE}.
Then, there exists a positive constant $C$ independent of $n$ such that
\begin{align*}
\sup_{t\in(0,\infty)}\left|
\tilde{f}(t) - \left[\frac{q + p (\E^t - 1)}{\E^t}
+\sum_{j=-n}^n\left\{\tilde{f}(\phi(jh)) - \frac{q + p\E^{\pi\sinh(jh)}}{1+\E^{\pi\sinh(jh)}}\right\}
S(j,h)(\phi^{-1}(t))
\right]
\right|
\leq C \E^{-\pi d n/\arsinh(d n/\mu)}.
\end{align*}
\end{theorem}

\subsection{Sinc indefinite integration}

Integrating both sides of~\eqref{eq:Sinc-approximation},
we have
\begin{equation}
 \int_{-\infty}^{\xi} F(x)\D x
\approx \sum_{j=-M}^N F(jh)\int_{-\infty}^{\xi}S(j,h)(x)\D x
=\sum_{j=-M}^N F(jh)J(j,h)(\xi),\quad\xi\in\mathbb{R},
\label{eq:Sinc-indef}
\end{equation}
where
\[
 J(j,h)(x) = h\left\{\frac{1}{2} + \frac{1}{\pi}\Si\left(\frac{\pi(x - jh)}{h}\right)\right\}.
\]
Here, $\Si(x)$ is the sine integral defined by~\eqref{eq:Si}.
The approximation~\eqref{eq:Sinc-indef} is called
the Sinc indefinite integration.
Similar to the Sinc approximation,
the Sinc indefinite integration is frequently combined with
a variable transformation.

\subsubsection{SE-Sinc indefinite integration and its convergence theorem}

In the case of the following integral
\[
 \int_0^t f(s)\D s, \quad t\in (0,\infty),
\]
where $f(s)$ decays exponentially as $s\to\infty$,
the SE transformation $s=\psi(x)$ allows us to apply
the Sinc indefinite integration~\eqref{eq:Sinc-indef} as
\[
 \int_0^t f(s)\D s
=\int_{-\infty}^{\psi^{-1}(t)} f(\psi(x))\psi'(x)\D x
\approx \sum_{j=-M}^N f(\psi(jh))\psi'(jh)J(j,h)(\psi^{-1}(t)),
\quad t\in(0,\infty).
\]
We refer to this approximation as the SE-Sinc indefinite integration.
Its convergence was analyzed as follows.

\begin{theorem}[Hara and Okayama~{\cite[Theorem~2]{HaraOkayamaIndef}}]
\label{thm:SE-Sinc-indef}
Assume that $f$ is analytic in $\psi(\domD_d)$ with $0<d<\pi$.
Furthermore, assume that there exists positive constants
$C_{\ddagger}$, $\beta$ and $\alpha$ with $0<\alpha\leq 1$ such that
\begin{equation}
|f(z)|\leq C_{\ddagger} \left|\frac{z}{1+z}\right|^{\alpha-1} |\E^{-z}|^{\beta}
\label{eq:f-bound-indef}
\end{equation}
holds for all $z\in\psi(\domD_d)$.
Let $\mu=\min\{\alpha,\beta\}$, let $M$ and $N$ be defined as~\eqref{eq:def-MN},
and let $h$ be defined as~\eqref{eq:def-h}.
Then, there exists a positive constant $C$ independent of $n$ such that
\[
 \sup_{t\in(0,\infty)}\left|
\int_0^t f(s)\D s - \sum_{j=-M}^N f(\psi(jh))\psi'(jh)J(j,h)(\psi^{-1}(t))
\right|\leq C \E^{-\sqrt{\pi d \mu n}}.
\]
\end{theorem}

\subsubsection{DE-Sinc indefinite integration and its convergence theorem}

In the case of the Sinc indefinite integration as well,
we may use the DE transformation instead of the SE transformation.
If the DE transformation is employed, we have
\[
 \int_0^t f(s)\D s
=\int_{-\infty}^{\phi^{-1}(t)} f(\phi(x))\phi'(x)\D x
\approx \sum_{j=-M}^N f(\phi(jh))\phi'(jh)J(j,h)(\phi^{-1}(t)),
\quad t\in(0,\infty).
\]
We refer to this approximation as the DE-Sinc indefinite integration.
Its convergence was analyzed as follows.

\begin{theorem}[Hara and Okayama~{\cite[Theorem~2]{HaraOkayama}}]
\label{thm:DE-Sinc-indef}
Assume that $f$ is analytic in $\phi(\domD_d)$ with $0<d<\pi/2$.
Furthermore, assume that there exists positive constants
$C_{\ddagger}$, $\beta$ and $\alpha$ with $0<\alpha\leq 1$
such that~\eqref{eq:f-bound-indef}
holds for all $z\in\psi(\domD_d)$.
Let $\mu=\min\{\alpha,\beta\}$, let $M$ and $N$ be defined as
\begin{equation}
\left\{
\begin{aligned}
M&=n,& N&=\left\lceil\frac{1}{h}\arsinh\left(\frac{\alpha}{\beta}\sinh(nh)\right)\right\rceil
 &&(\text{if}\,\,\,\mu = \alpha),\\
N&=n,& M&=\left\lceil\frac{1}{h}\arsinh\left(\frac{\beta}{\alpha}\sinh(nh)\right)\right\rceil
 &&(\text{if}\,\,\,\mu = \beta),
\end{aligned}
\label{eq:def-MN-DE}
\right.
\end{equation}
and let $h$ be defined as~\eqref{eq:def-h-DE}.
Then, there exists a positive constant $C$ independent of $n$ such that
\[
 \sup_{t\in(0,\infty)}\left|
\int_0^t f(s)\D s - \sum_{j=-M}^N f(\phi(jh))\phi'(jh)J(j,h)(\phi^{-1}(t))
\right|\leq C \frac{\arsinh(dn/\mu)}{n} \E^{-\pi d n/\arsinh(dn/\mu)}.
\]
\end{theorem}

\section{Sinc-Nystr\"om methods} \label{sec:SNM}

In this section, we describe two Sinc-Nystr\"{o}m methods
and their error analyses presented in Hara and Okayama~\cite{HaraOkayama}.
The derivation of the methods is through the following two steps:
(i) by integration, rewrite the given problem~\eqref{eq:ODE} as
\begin{align}
\mathbd{y}(t)=\mathbd{r}+\int_{0}^{t}\{K(s)\mathbd{y}(s)+\mathbd{g}(s)\}\D s,
\label{eq:y(t)}
\end{align}
and (ii) apply the SE-Sinc indefinite integration
or the DE-Sinc indefinite integration to the integral.
We explain these methods individually.

\subsection{SE-Sinc-Nystr\"{o}m method and its error analysis} \label{sec:SE-SNM}

Let $l=M+N+1$ and let $\mathbd{y}^{(l)}(t)$ be an approximate solution of
$\mathbd{y}(t)$.
Approximating the integral in~\eqref{eq:y(t)} based on
Theorem~\ref{thm:SE-Sinc-indef}, we can derive
\begin{align}
\mathbd{y}^{(l)}(t)
=\mathbd{r}+\sum_{j=-M}^N\Bigl\{K(\psi(jh))\mathbd{y}^{(l)}(\psi(jh))
+\mathbd{g}(\psi(jh))\Bigr\}\psi'(jh)J(j,h)(\psi^{-1}(t)). \label{eq:ynMMS}
\end{align}
To determine the unknown coefficients $\mathbd{y}^{(l)}(\psi(jh))$, we set
sampling points at $t=\psi(ih)$\, $(i=-M,\,-M+1,\,\ldots,\,N)$.
We then obtain a system of linear equations given by
\begin{align}
(I_m\otimes I_l-\{I_m\otimes(hI_l^{(-1)}D^{(\psi)}_l)\}[K_{ij}^{(\psi)}])\mathbd{Y}^{(\psi)}
=\mathbd{R}+\{I_m\otimes (hI_l^{(-1)}D^{(\psi)}_l)\}\mathbd{G}^{(\psi)},
 \label{renritsuMMS}
\end{align}
where $I_l$ and $I_m$ are identity matrices, $\otimes$ denotes the Kronecker
product, and $I^{(-1)}_l$ is an $l\times l$ matrix whose $(i,j)$ entries
are defined as
\begin{align*}
(I^{(-1)}_{l})_{ij}=\frac{1}{2}+\frac{1}{\pi}\Si\left(\pi(i-j)\right)
\quad(i,j=-M,\,-M+1,\,\dots,\,N).
\end{align*}
Moreover, $D^{(\psi)}_{l}$ and $K^{(\psi)}_{ij}$ are $l\times l$ diagonal
matrices defined as
\begin{align*}
D^{(\psi)}_{l}&=\diag[\psi'(-Mh),\,\dots,\,\psi'(Nh)], \\
K^{(\psi)}_{ij}&=\diag[k_{ij}(\psi(-Mh)),\,\dots,\,k_{ij}(\psi(Nh))],
\end{align*}
and $[K^{(\psi)}_{ij}]$ is a block matrix whose $(i,j)$ entry is
$K^{(\psi)}_{ij}~(i,j=1,\,\dots,\,m)$. Furthermore, $\mathbd{R}$,
$\mathbd{Y}^{(\psi)}$, and $\mathbd{G}^{(\psi)}$ are $lm$-dimensional
vectors defined as follows:
\begin{align*}
\mathbd{R}&=[r_{1},\,\dots,\,r_{1},\,r_{2},\,\dots,\,r_{2},\,\dots,\,r_{m},\,\dots,\,r_{m}]^{\mathrm{T}},\\
\mathbd{Y}^{(\psi)}&=[y^{(l)}_{1}(\psi(-Mh)),\,\dots,\,y^{(l)}_{1}(\psi(Nh)),
y^{(l)}_{2}(\psi(-Mh)),\,\dots,\,y^{(l)}_{2}(\psi(Nh)),
\dots,\,y^{(l)}_{m}(\psi(-Mh)),\,\dots,\,y^{(l)}_{m}(\psi(Nh))]^{\mathrm{T}},\\
\mathbd{G}^{(\psi)}&=[g_{1}(\psi(-Mh)),\,\dots,\,g_{1}(\psi(Nh)),
g_{2}(\psi(-Mh)),\,\dots,\,g_{2}(\psi(Nh)),
\dots,\,g_{m}(\psi(-Mh)),\,\dots,\,g_{m}(\psi(Nh))]^{\mathrm{T}}.
\end{align*}
By solving~\eqref{renritsuMMS}, we can obtain the value of
$\mathbd{y}^{(l)}(\psi(jh))$, from which $\mathbd{y}^{(l)}(t)$ is
determined through~\eqref{eq:ynMMS}. This procedure is the
SE-Sinc-Nystr\"{o}m method.
Its error was analyzed as follows.

\begin{theorem}[Hara and Okayama~{\cite[Theorem~3]{HaraOkayama}}]
\label{thm:SE-Sinc-Nystroem}
Let $\beta$ be a positive constant,
and let $\alpha$ and $d$ be constants with $0 <\alpha\leq 1$ and
$0 < d< \pi$.
Assume that the function $k_{ij}$
$(i, j = 1,\,\dots,\,m)$ is analytic and bounded on $\psi(\domD_d)$,
and $y_i$ and $g_i$ $(i = 1,\,\dots,\,m)$
satisfy the assumption of Theorem~\ref{thm:SE-Sinc-indef}.
Let $h$ be set as~\eqref{eq:def-h}, and let $M$ and $N$
be set as~\eqref{eq:def-MN}.
Let $A_{lm}$ be a coefficient matrix of
the system of linear equation~\eqref{renritsuMMS}, i.e.,
\[
 A_{lm}
 = (I_m\otimes I_l-\{I_m\otimes(hI_l^{(-1)}D^{(\psi)}_l)\}[K_{ij}^{(\psi)}]),
\]
and assume that the inverse matrix of $A_{lm}$ exists.
Then, the error of the approximate solution $\mathbd{y}^{(l)}(t)$
in~\eqref{eq:ynMMS}
is estimated as
\[
 \max_{1\leq i\leq m}
\left\{
\sup_{t\in(0,\infty)}
\left|y_i(t) - y^{(l)}_i(t)\right|
\right\}
\leq \left(C + \hat{C}\|A_{lm}^{-1}\|_{\infty}\right)
\sqrt{n}\E^{-\sqrt{\pi d \mu n}},
\]
where $C$ and $\hat{C}$ are positive constants independent of $n$.
\end{theorem}

\subsection{DE-Sinc-Nystr\"{o}m method and its error analysis} \label{sec:DE-SNM}

From a comparison of Theorem~\ref{thm:SE-Sinc-indef} with
Theorem~\ref{thm:DE-Sinc-indef},
we can expect that
replacement of $\psi$ by $\phi$ may accelerate the convergence rate.
Approximating the integral in~\eqref{eq:y(t)} based on
Theorem~\ref{thm:DE-Sinc-indef}, we can derive
\begin{align}
\mathbd{y}^{(l)}(t)
=\mathbd{r}+\sum_{j=-M}^N\Bigl\{K(\phi(jh))\mathbd{y}^{(l)}(\phi(jh))
+\mathbd{g}(\phi(jh))\Bigr\}\phi'(jh)J(j,h)(\phi^{-1}(t)). \label{eq:ynMMD}
\end{align}
Setting sampling points at $t=\phi(ih)$ $(i=-M,\,-M+1,\,\dots,\,N)$,
we obtain the following:
\begin{align}
(I_m\otimes I_l-\{I_m\otimes(hI_l^{(-1)}D^{(\phi)}_l)\}[K_{ij}^{(\phi)}])\mathbd{Y}^{(\phi)}
=\mathbd{R}+\{I_m\otimes(hI_l^{(-1)}D^{(\phi)}_l)\}\mathbd{G}^{(\phi)},
\label{renritsuMMD}
\end{align}
for which $\phi$ is used instead of $\psi$. By solving~\eqref{renritsuMMD},
we can obtain the value of $\mathbd{y}^{(l)}(\phi(jh))$, from which
$\mathbd{y}^{(l)}(t)$ is determined  through~\eqref{eq:ynMMD}.
This procedure is the DE-Sinc-Nystr\"{o}m method.
Its error was analyzed as follows.

\begin{theorem}[Hara and Okayama~{\cite[Theorem~4]{HaraOkayama}}]
\label{thm:DE-Sinc-Nystroem}
Let $\beta$ be a positive constant,
and let $\alpha$ and $d$ be constants with $0 <\alpha\leq 1$ and
$0 < d< \pi/2$.
Assume that the function $k_{ij}$
$(i, j = 1,\,\dots,\,m)$ is analytic and bounded on $\phi(\domD_d)$,
and $y_i$ and $g_i$ $(i = 1,\,\dots,\,m)$
satisfy the assumption of Theorem~\ref{thm:DE-Sinc-indef}.
Let $h$ be set as~\eqref{eq:def-h-DE}, and let $M$ and $N$
be set as~\eqref{eq:def-MN-DE}.
Let $B_{lm}$ be a coefficient matrix of
the system of linear equation~\eqref{renritsuMMD}, i.e.,
\[
 B_{lm}
 = (I_m\otimes I_l-\{I_m\otimes(hI_l^{(-1)}D^{(\phi)}_l)\}[K_{ij}^{(\phi)}]),
\]
and assume that the inverse matrix of $B_{lm}$ exists.
Then, the error of the approximate solution $\mathbd{y}^{(l)}(t)$
in~\eqref{eq:ynMMD}
is estimated as
\[
 \max_{1\leq i\leq m}
\left\{
\sup_{t\in(0,\infty)}
\left|y_i(t) - y^{(l)}_i(t)\right|
\right\}
\leq \left(C + \hat{C}\|B_{lm}^{-1}\|_{\infty}\right)
\arsinh(dn/\mu)\E^{-\pi d n/\arsinh(dn/\mu)},
\]
where $C$ and $\hat{C}$ are positive constants independent of $n$.
\end{theorem}

\section{Sinc-collocation methods} \label{sec:SCM}

In this section, we describe two Sinc-collocation methods
presented in Okayama and Hara~\cite{OkayamaHara},
and their error analyses presented in this paper.
The derivation of the methods is through
applying the SE- or DE-Sinc approximation with a boundary treatment
to the approximate solution $\mathbd{y}^{(l)}(t)$.

\subsection{SE-Sinc-collocation method and its error analysis (new result)} \label{sec:SE-SCM}

Here, we approximate $\mathbd{y}^{(l)}$ in~\eqref{eq:ynMMS}
based on Theorem~\ref{thm:SE-Sinc-boundary}.
Then,
we obtain a new approximate solution $\hat{\mathbd{y}}^{(l)}$ defined as
\begin{align}
\hat{\mathbd{y}}^{(l)}(t)
=
\frac{\mathbd{r}+\mathbd{p}^{(\psi)}(\E^t - 1)}{\E^t}
+\sum_{k=-M}^N\Biggl\{ \mathbd{y}^{(l)}(\psi(kh))
-\frac{\mathbd{r}+\mathbd{p}^{(\psi)}\E^{kh}}{1+\E^{kh}}\Biggr\}
S(k,h)(\phi^{-1}(t)),
\label{eq:ynSE-SCM}
\end{align}
where $\mathbd{r}$ is the initial value in~\eqref{eq:ODE}, and
$\mathbd{p}^{(\psi)}$ is given by
\[
 \mathbd{p}^{(\psi)}
=\mathbd{r}+
h\sum_{j=-M}^N \left\{K(\psi(jh))\mathbd{y}^{(l)}(\psi(jh))
+\mathbd{g}(\psi(jh))\right\}\psi'(jh).
\]
Here, $J(j,h)(\psi^{-1}(0))=0$ and
$\lim_{t\to \infty}J(j,h)(\psi^{-1}(t))=h$ are used
to calculate $\mathbd{y}^{(l)}(0)$
and $\lim_{t\to\infty}\mathbd{y}^{(l)}(t)$,
respectively, in accordance with
the definition in~\eqref{eq:ynMMS}.
In summary, by solving~\eqref{renritsuMMS}, we can obtain the value of
$\mathbd{y}^{(l)}(\psi(jh))$, from which $\hat{\mathbd{y}}^{(l)}(t)$ is
determined through~\eqref{eq:ynSE-SCM}. This procedure is the
SE-Sinc-collocation method.
We provide its error analysis as follows.
The proof is given in Sect.~\ref{sec:proofs}.

\begin{theorem}
\label{thm:SE-Sinc-collocation}
Assume that the assumptions of Theorem~\ref{thm:SE-Sinc-Nystroem}
are fulfilled.
Furthermore, assume that there exists a positive constant $H$
such that
\begin{align*}
\max_{i=1,\,\ldots,\,m}|y_i(z) - r_i|
 &\leq H \left|\frac{z}{1+z}\right|^{\alpha},\\
\max_{i=i,\,\ldots,\,m}|y_i(z)|
 &\leq H \left|\E^{-z}\right|^{\beta}
\end{align*}
hold for all $z\in\psi(\domD_d)$.
Then, the error of the approximate solution $\hat{\mathbd{y}}^{(l)}(t)$
in~\eqref{eq:ynSE-SCM}
is estimated as
\[
 \max_{1\leq i\leq m}
\left\{
\sup_{t\in(0,\infty)}
\left|y_i(t) - \hat{y}^{(l)}_i(t)\right|
\right\}
\leq \left(C + \hat{C}\|A_{lm}^{-1}\|_{\infty}\right)
\sqrt{n}\log(n+1)\E^{-\sqrt{\pi d \mu n}},
\]
where $C$ and $\hat{C}$ are positive constants independent of $n$.
\end{theorem}

\subsection{DE-Sinc-collocation method and its error analysis (new result)} \label{sec:DE-SCM}

Here, we approximate $\mathbd{y}^{(l)}$ in~\eqref{eq:ynMMD}
based on Theorem~\ref{thm:DE-Sinc-boundary}.
Then,
we obtain a new approximate solution $\hat{\mathbd{y}}^{(l)}$ defined as
\begin{align}
\hat{\mathbd{y}}^{(l)}(t)
=
\frac{\mathbd{r}+\mathbd{p}^{(\phi)}(\E^t - 1)}{\E^t}
+\sum_{k=-M}^N\Biggl\{ \mathbd{y}^{(l)}(\phi(kh))
-\frac{\mathbd{r}+\mathbd{p}^{(\phi)}\E^{\pi\sinh(kh)}}{1+\E^{\pi\sinh(kh)}}\Biggr\}
S(k,h)(\phi^{-1}(t)),
\label{eq:ynDE-SCM}
\end{align}
where $\mathbd{r}$ is the initial value in~\eqref{eq:ODE}, and
$\mathbd{p}^{(\phi)}$ is given by
\[
 \mathbd{p}^{(\phi)}
=\mathbd{r}+
h\sum_{j=-M}^N \left\{K(\phi(jh))\mathbd{y}^{(l)}(\phi(jh))
+\mathbd{g}(\phi(jh))\right\}\phi'(jh).
\]
In summary, by solving~\eqref{renritsuMMD}, we can obtain the value of
$\mathbd{y}^{(l)}(\phi(jh))$, from which $\hat{\mathbd{y}}^{(l)}(t)$ is
determined through~\eqref{eq:ynDE-SCM}. This procedure is the
DE-Sinc-collocation method.
We provide its error analysis as follows.
The proof is given in Sect.~\ref{sec:proofs}.

\begin{theorem}
\label{thm:DE-Sinc-collocation}
Assume that the assumptions of Theorem~\ref{thm:DE-Sinc-Nystroem}
are fulfilled.
Furthermore, assume that there exists a positive constant $H$
such that
\begin{align*}
\max_{i=1,\,\ldots,\,m}|y_i(z) - r_i|
 &\leq H \left|z\right|^{\alpha},\\
\max_{i=i,\,\ldots,\,m}|y_i(z)|
 &\leq H \left|\E^{-z}\right|^{\beta}
\end{align*}
hold for all $z\in\phi(\domD_d)$.
Then, the error of the approximate solution $\hat{\mathbd{y}}^{(l)}(t)$
in~\eqref{eq:ynDE-SCM}
is estimated as
\[
 \max_{1\leq i\leq m}
\left\{
\sup_{t\in(0,\infty)}
\left|y_i(t) - \hat{y}^{(l)}_i(t)\right|
\right\}
\leq \left(C + \hat{C}\|B_{lm}^{-1}\|_{\infty}\right)
\arsinh(dn/\mu)\log(n+1)\E^{-\pi d n/\arsinh(dn/\mu)},
\]
where $C$ and $\hat{C}$ are positive constants independent of $n$.
\end{theorem}

\section{Proof of Theorems~\ref{thm:DE-Sinc} and~\ref{thm:DE-Sinc-boundary}}
\label{sec:proofs-DE-Sinc}

In this section, we prove Theorems~\ref{thm:DE-Sinc}
and~\ref{thm:DE-Sinc-boundary}.

\subsection{Sketch of the proof of Theorem~\ref{thm:DE-Sinc}}

We prove Theorem~\ref{thm:DE-Sinc}
with an explicit form of the constant $C$, i.e.,
we prove the following theorem.

\begin{theorem}
\label{thm:DE-Sinc-with-constant}
Assume that $f$ is analytic in $\phi(\domD_d)$ with $0<d<\pi/2$.
Furthermore, assume that
there exist positive constants $C_{\ddagger}$ and $\mu$ with $\mu\leq 1$
such that~\eqref{eq:f-bound-DE}
holds for all $z\in\phi(\domD_d)$.
Let $h$ be defined as~\eqref{eq:def-h-DE}.
Then, it holds that
\[
 \sup_{t\in(0,\infty)}\left|
  f(t) - \sum_{j=-n}^n f(\phi(jh))S(j,h)(\phi^{-1}(t))
 \right|
\leq C_{\dagger} \E^{-\pi d n/\arsinh(dn/\mu)},
\]
where $C_{\dagger}$ is a positive constant defined as
\begin{equation}
 C_{\dagger} = \frac{2 C_{\ddagger}}{\mu\pi^{1-\mu}}
\left[
\frac{2}{\pi d(1 - \E^{-2\pi d/\arsinh(d/\mu)})\cos^{2\mu}((\pi/2)\sin d)\cos^{1+\mu} d}
+ \frac{\E^{-\pi d (\arsinh(d/\mu) - 1)/\arsinh(d/\mu)}}{\arsinh(d/\mu)(1 + (d/\mu))^{(1-\mu)/2}}
\right].
\label{eq:Cast}
\end{equation}
\end{theorem}

The proof is organized as follows.
We divide the error into two terms as
\begin{align*}
 \left|
  f(t) - \sum_{j=-n}^n f(\phi(jh))S(j,h)(\phi^{-1}(t))
 \right|
&\leq \left|
  f(t) - \sum_{j=-\infty}^{\infty} f(\phi(jh))S(j,h)(\phi^{-1}(t))
 \right|\\
&\quad +\left|
 \sum_{j=-\infty}^{-n-1} f(\phi(jh))S(j,h)(\phi^{-1}(t))
+\sum_{j=n+1}^{\infty} f(\phi(jh))S(j,h)(\phi^{-1}(t))
 \right|.
\end{align*}
The first term is called the discretization error,
and the second term is called the truncation error.
The discretization error was estimated as follows.

\begin{lemma}[Okayama~{\cite[Lemma~4.16]{OkayamaSinc}}]
\label{lem:discretization}
Assume that $f$ is analytic in $\phi(\domD_d)$ with $0<d<\pi/2$.
Furthermore, assume that
there exist positive constants $C_{\ddagger}$ and $\mu$ with $\mu\leq 1$
such that~\eqref{eq:f-bound-DE}
holds for all $z\in\phi(\domD_d)$.
Then, we have
\[
\sup_{t\in(0,\infty)} \left|
  f(t) - \sum_{j=-\infty}^{\infty} f(\phi(jh))S(j,h)(\phi^{-1}(t))
 \right|
\leq \frac{4 C_{\ddagger} \E^{-\pi d/h}}{\pi^{2-\mu}d \mu (1 - \E^{-2\pi d/h})\cos^{2\mu}((\pi/2)\sin d)\cos^{1+\mu} d}.
\]
\end{lemma}

We estimate the truncation error as follows.
The proof is given in Sect.~\ref{sec:truncation}.

\begin{lemma}
\label{lem:truncation}
Assume that
there exist positive constants $C_{\ddagger}$ and $\mu$ with $\mu\leq 1$
such that~\eqref{eq:f-bound-DE}
holds for all $z\in\phi(\domD_d)$.
Then, we have
\[
\sup_{t\in(0,\infty)}\left|
 \sum_{j=-\infty}^{-n-1} f(\phi(jh))S(j,h)(\phi^{-1}(t))
+\sum_{j=n+1}^{\infty} f(\phi(jh))S(j,h)(\phi^{-1}(t))
 \right|
\leq \frac{2 C_{\ddagger}\pi^{\mu-1} \E^{-\pi\mu\sinh(nh)}}{\mu h \cosh^{1-\mu}(nh)}.
\]
\end{lemma}

Combining these two lemmas, we prove Theorem~\ref{thm:DE-Sinc-with-constant}
in Sect.~\ref{sec:DE-Sinc-with-constant}.
This completes the proof.

\subsection{Proof of Lemma~\ref{lem:truncation}}
\label{sec:truncation}

To prove Lemma~\ref{lem:truncation},
the following inequality is useful.

\begin{lemma}[Okayama~{\cite[Lemma~4.15]{OkayamaSinc}}]
\label{lem:log-bound}
For all real numbers $x$ and $y$ with $|y|<\pi/2$,
we have
\[
 |\log(1 + \E^{\pi\sinh(x+\I y)})|
\leq \frac{1}{\cos((\pi/2)\sin y)\cos y}\cdot
\frac{\pi\cosh x}{1+\E^{-\pi\sinh(x)\cos y}}.
\]
\end{lemma}

Using this lemma, we prove Lemma~\ref{lem:truncation} as follows.

\begin{proof}
Put $F(x)=f(\phi(x))$ for short.
Because $|S(j,h)(x)|\leq 1$ for all $x\in\mathbb{R}$,
it holds that
\begin{align}
 \left|
 \sum_{j=-\infty}^{-n-1} F(jh)S(j,h)(\phi^{-1}(t))
+\sum_{j=n+1}^{\infty} F(jh)S(j,h)(\phi^{-1}(t))
 \right|
&\leq \sum_{j=-\infty}^{-n-1} |F(jh)||S(j,h)(\phi^{-1}(t))|
+\sum_{j=n+1}^{\infty} |F(jh)||S(j,h)(\phi^{-1}(t))|\nonumber\\
&\leq \sum_{j=-\infty}^{-n-1} |F(jh)|
+\sum_{j=n+1}^{\infty} |F(jh)|.
\label{leq:truncation-first}
\end{align}
Furthermore, from~\eqref{eq:f-bound-DE}, using Lemma~\ref{lem:log-bound}
with $y=0$, we have
\[
 |F(x)|=|f(\phi(x))|\leq C_{\ddagger} |\log(1 + \E^{\pi\sinh x})|^{\mu}
\left|\frac{1}{1+\E^{\pi\sinh x}}\right|^{\mu}
\leq C_{\ddagger}
\frac{(\pi\cosh x)^{\mu}}{(1+\E^{-\pi\sinh x})^{\mu}(1+\E^{\pi\sinh x})^{\mu}}.
\]
Therefore, the right-hand side of~\eqref{leq:truncation-first}
is further bounded as
\begin{align*}
 \sum_{j=-\infty}^{-n-1} |F(jh)|
+\sum_{j=n+1}^{\infty} |F(jh)|
&\leq\sum_{j=-\infty}^{-n-1}
C_{\ddagger}
\frac{(\pi\cosh(jh))^{\mu}}{(1+\E^{-\pi\sinh(jh)})^{\mu}(1+\E^{\pi\sinh(jh)})^{\mu}}
+\sum_{j=n+1}^{\infty}
C_{\ddagger}
\frac{(\pi\cosh(jh))^{\mu}}{(1+\E^{-\pi\sinh(jh)})^{\mu}(1+\E^{\pi\sinh(jh)})^{\mu}}\\
&=\frac{2C_{\ddagger}}{h}\cdot h\sum_{j=n+1}^{\infty}
\frac{\pi^{\mu}\cosh^{\mu}(jh)\E^{-\pi\mu\sinh(jh)}}{(1+\E^{-\pi\sinh(jh)})^{2\mu}}\\
&\leq\frac{2C_{\ddagger}}{h}\cdot h\sum_{j=n+1}^{\infty}
\frac{\pi^{\mu}\cosh^{\mu}(jh)\E^{-\pi\mu\sinh(jh)}}{(1+0)^{2\mu}}\\
&\leq \frac{2C_{\ddagger}}{h}\int_{nh}^{\infty}
\pi^{\mu}\cosh^{\mu} x\E^{-\pi\mu\sinh x}\D x.
\end{align*}
Here, because $0<\mu\leq 1$, it holds for $x\geq nh$ that
\[
 \cosh^{\mu} x = \frac{\cosh x}{\cosh^{1-\mu} x}
\leq \frac{\cosh x}{\cosh^{1-\mu}(nh)},
\]
from which we have
\begin{align*}
\frac{2C_{\ddagger}}{h}\int_{nh}^{\infty}
\pi^{\mu}\cosh^{\mu} x\E^{-\pi\mu\sinh x}\D x
\leq\frac{2C_{\ddagger}}{h}\int_{nh}^{\infty}
\pi^{\mu}\frac{\cosh x}{\cosh^{1-\mu}(nh)}\E^{-\pi\mu\sinh x}\D x
=\frac{2C_{\ddagger}\pi^{\mu-1}}{\mu h\cosh^{1-\mu}(nh)}\E^{-\pi\mu\sinh(nh)}.
\end{align*}
This is the desired estimate.
\end{proof}

\subsection{Proof of Theorem~\ref{thm:DE-Sinc-with-constant}}
\label{sec:DE-Sinc-with-constant}

To prove Theorem~\ref{thm:DE-Sinc-with-constant},
we prepare the following three propositions.

\begin{proposition}[Okayama and Kawai~{\cite[Proposition~7]{OkayamaKawai}}]
\label{prop:OkayamaKawai}
Let $\tilde{q}$ be a function defined by
\[
 \tilde{q}(x) = \frac{x}{\arsinh x}.
\]
Then, $\tilde{q}(x)$ monotonically increases for $x\geq 0$.
\end{proposition}
\begin{proposition}
\label{prop:p}
Let $\tilde{p}$ be a function defined by
\[
 \tilde{p}(x)
 = \frac{\arsinh x}{x}\sqrt{1 + x^2}.
\]
Then, $\tilde{p}(x)$ monotonically decreases for $x\geq 0$.
\end{proposition}
\begin{proof}
Putting $t=\arsinh x$, we have $\tilde{p}(x)=\tilde{r}(t)$, where
\[
 \tilde{r}(t)=\frac{t}{\tanh t}.
\]
Differentiating $\tilde{r}(t)$, we have
\[
 \tilde{r}'(t) = \frac{\cosh t\sinh t - t}{\sinh^2 t}
\geq \frac{1\cdot\sinh t - t}{\sinh^2 t}\geq 0
\]
for $t\geq 0$. Thus, it holds for $x\geq 0$ that
\[
 \tilde{p}'(x)=\tilde{r}'(\arsinh x)\left(\arsinh x\right)'
=\tilde{r}'(\arsinh x)\frac{1}{\sqrt{1+x^2}}\geq 0,
\]
which is to be demonstrated.
\end{proof}
\begin{proposition}
\label{prop:w}
Let $w$ be a function defined by
\[
 w(x) = (1 + x^2)\exp\left\{-2\pi x\left(1 - \frac{1}{\arsinh x}\right)\right\}.
\]
Then, $w(x)$ monotonically decreases for $x\geq 0$.
\end{proposition}
\begin{proof}
Differentiating $w(x)$, we have
\begin{align*}
 w'(x) &= -2
\left\{x + \frac{\pi(1+x^2)}{\arsinh^2 x}
\left(\arsinh^2 x - \arsinh x + \frac{x}{\sqrt{1+x^2}}\right)
\right\}\exp\left\{-2\pi x\left(1 - \frac{1}{\arsinh x}\right)\right\}\\
&=-2
\left\{x + \frac{\pi(1+x^2)}{\arsinh^2 x}
v(\arsinh x)
\right\}\exp\left\{-2\pi x\left(1 - \frac{1}{\arsinh x}\right)\right\},
\end{align*}
where $v(t) = t^2 - t +\tanh t$.
Differentiating $v(t)$, we have
\begin{align*}
 v'(t) &= 2 t - 1 + \frac{1}{\cosh^2 t},\\
 v''(t)&= 2\left(\frac{\cosh^3 t - \sinh t}{\cosh^3 t}\right)
\geq 2\left(\frac{\cosh t - \sinh t}{\cosh^3 t}\right)
= 2\left(\frac{\E^{-t}}{\cosh^3 t}\right)
> 0.
\end{align*}
Therefore, $v'(t)$ monotonically increases, from which we have
$v'(t)\geq v'(0) = 0$ for $t\geq 0$.
Therefore, $v(t)$ also monotonically increases for $t\geq 0$,
from which we have $v(t)\geq v(0)= 0$ for $t\geq 0$.
Thus, it holds for $x\geq 0$ that
\[
  w'(x)=-2
\left\{x + \frac{\pi(1+x^2)}{\arsinh^2 x}
v(\arsinh x)
\right\}\exp\left\{-2\pi x\left(1 - \frac{1}{\arsinh x}\right)\right\}
\leq 0,
\]
which is to be demonstrated.
\end{proof}

Using these propositions,
we prove Theorem~\ref{thm:DE-Sinc-with-constant} as follows.

\begin{proof}
From Lemmas~\ref{lem:discretization} and~\ref{lem:truncation},
we have
\begin{align*}
& \left|f(t) - \sum_{j=-n}^n f(\phi(jh))S(j,h)(\phi^{-1}(t))\right|\\
&\leq \frac{4 C_{\ddagger}}{\pi^{2-\mu}d \mu (1 - \E^{-2\pi d/h})\cos^{2\mu}((\pi/2)\sin d)\cos^{1+\mu} d}
\E^{-\pi d/h}
+\frac{2 C_{\ddagger}\pi^{\mu-1}}{\mu h \cosh^{1-\mu}(nh)}
\E^{-\pi\mu\sinh(nh)}.
\end{align*}
As for the first term,
substituting~\eqref{eq:def-h-DE} and using $\tilde{q}(x)$
defined in Proposition~\ref{prop:OkayamaKawai}, we have
\begin{align*}
&\frac{4 C_{\ddagger}}{\pi^{2-\mu}d \mu (1 - \E^{-2\pi d/h})\cos^{2\mu}((\pi/2)\sin d)\cos^{1+\mu} d}
\E^{-\pi d/h}\\
&=\frac{4 C_{\ddagger}}{\pi^{2-\mu}d \mu (1 - \E^{-2\pi \mu \tilde{q}(dn/\mu)})\cos^{2\mu}((\pi/2)\sin d)\cos^{1+\mu} d}
\E^{-\pi d n/\arsinh(dn/\mu)}\\
&\leq \frac{4 C_{\ddagger}}{\pi^{2-\mu}d \mu (1 - \E^{-2\pi \mu \tilde{q}(d\cdot 1/\mu)})\cos^{2\mu}((\pi/2)\sin d)\cos^{1+\mu} d}
\E^{-\pi d n/\arsinh(dn/\mu)}.
\end{align*}
As for the second term,
substituting~\eqref{eq:def-h-DE} and using $\tilde{p}(x)$
and $w(x)$ defined in Propositions~\ref{prop:p} and~\ref{prop:w},
respectively, we have
\begin{align*}
 \frac{2 C_{\ddagger}\pi^{\mu-1}}{\mu h \cosh^{1-\mu}(nh)}
\E^{-\pi\mu\sinh(nh)}
=\frac{2 C_{\ddagger}\pi^{\mu-1}}{d \tilde{p}(dn/\mu)}
\left\{w(dn/\mu)\right\}^{\mu/2}\E^{-\pi d n/\arsinh(dn/\mu)}
\leq \frac{2 C_{\ddagger}\pi^{\mu-1}}{d \tilde{p}(d\cdot 1/\mu)}
\left\{w(d\cdot 1/\mu)\right\}^{\mu/2}\E^{-\pi d n/\arsinh(dn/\mu)}.
\end{align*}
Thus, the claim follows.
\end{proof}

\subsection{Proof of Theorem~\ref{thm:DE-Sinc-boundary} and its improvement}

Note that the approximation~\eqref{eq:general-DE-Sinc}
is equivalent to~\eqref{eq:DE-Sinc-approximation}
if we define $f$ as~\eqref{eq:Hamada-func}.
Therefore, from Theorem~\ref{thm:DE-Sinc-with-constant},
we can readily deduce the following theorem,
which proves Theorem~\ref{thm:DE-Sinc-boundary}.

\begin{theorem}
\label{thm:naive}
For a given function $\tilde{f}$,
let $p$ and $q$ be defined by~\eqref{eq:def-p-and-q},
and let $f$ be defined by~\eqref{eq:Hamada-func}.
Assume that $f$ is analytic in $\phi(\domD_d)$ with $0<d<\pi/2$.
Furthermore, assume that
there exist positive constants $C_{\ddagger}$ and $\mu$ with $\mu\leq 1$
such that~\eqref{eq:f-bound-DE}
holds for all $z\in\phi(\domD_d)$.
Let $h$ be defined as~\eqref{eq:def-h-DE}.
Then, it holds that
\begin{align*}
\sup_{t\in(0,\infty)}\left|
\tilde{f}(t) - \left[\frac{q + p (\E^t - 1)}{\E^t}
+\sum_{j=-n}^n\left\{\tilde{f}(\phi(jh)) - \frac{q + p\E^{\pi\sinh(jh)}}{1+\E^{\pi\sinh(jh)}}\right\}
S(j,h)(\phi^{-1}(t))
\right]
\right|
\leq C_{\dagger} \E^{-\pi d n/\arsinh(d n/\mu)},
\end{align*}
where $C_{\dagger}$ is a positive constant defined as~\eqref{eq:Cast}.
\end{theorem}

However, in this theorem, the conditions to be satisfied
are described not for a given function $\tilde{f}$,
but for a function $f$ defined by~\eqref{eq:Hamada-func}.
Such conditions are inconvenient to check whether those conditions
are satisfied or not.
Therefore, we provide a sufficient condition as follows.

\begin{lemma}
\label{lem:hoelder-DE}
For a given function $\tilde{f}$,
let $p$ and $q$ be defined by~\eqref{eq:def-p-and-q}.
Assume that $\tilde{f}$ is analytic in $\phi(\domD_d)$ with $0<d<\pi/2$.
Furthermore,
assume that there exists positive constants $H$ and $\mu$ with $\mu\leq 1$
such that
\begin{align*}
 |\tilde{f}(z) - q|&\leq H |z|^{\mu},\\
 |\tilde{f}(z) - p|&\leq H |\E^{-z}|^{\mu}
\end{align*}
hold for all $z\in\phi(\domD_d)$.
Then, $f$ defined by~\eqref{eq:Hamada-func}
satisfies the assumptions of Theorem~\ref{thm:DE-Sinc-with-constant}.
\end{lemma}

To prove this lemma,
we prepare the following four lemmas.

\begin{lemma}[Okayama et al.~{\cite[Lemma~4.22]{Okayama-et-al}}]
\label{lem:Lemma422}
Let $x$ and $y$ be real numbers with $|y|< \pi/2$,
and let $\zeta=x+\I y$. Then,
\begin{align*}
 \left|\frac{1}{1+\E^{\pi\sinh\zeta}}\right|
&\leq \frac{1}{(1+\E^{\pi\sinh(x)\cos y})\cos((\pi/2)\sin y)},\\
 \left|\frac{1}{1+\E^{-\pi\sinh\zeta}}\right|
&\leq \frac{1}{(1+\E^{-\pi\sinh(x)\cos y})\cos((\pi/2)\sin y)}.
\end{align*}
\end{lemma}
\begin{lemma}
\label{lem:rewrite-Lemma422}
Let $d$ be a positive constant with $d<\pi/2$.
Then,
\begin{align*}
 \sup_{z\in\phi(\domD_d)}\left|\E^{-z}\right|
&\leq \frac{1}{\cos((\pi/2)\sin d)},\\
 \sup_{z\in\phi(\domD_d)}\left|1 - \E^{-z}\right|
&\leq \frac{1}{\cos((\pi/2)\sin d)}.
\end{align*}
\end{lemma}
\begin{proof}
Applying $z=\phi(\zeta)$, we have
\begin{align*}
 \sup_{z\in\phi(\domD_d)}\left|\E^{-z}\right|
&=\sup_{\zeta\in\domD_d}\left|\E^{-\phi(\zeta)}\right|
=\sup_{\zeta\in\domD_d}\left|\frac{1}{1+\E^{\pi\sinh\zeta}}\right|,\\
 \sup_{z\in\phi(\domD_d)}\left|1 - \E^{-z}\right|
&=\sup_{\zeta\in\domD_d}\left|1 - \E^{-\phi(\zeta)}\right|
=\sup_{\zeta\in\domD_d}\left|\frac{1}{1+\E^{-\pi\sinh\zeta}}\right|.
\end{align*}
Thus, the claim follows from Lemma~\ref{lem:Lemma422}.
\end{proof}

\begin{lemma}
\label{lem:Maki}
Let $d$ be a positive constant with $d<\pi/2$. Then,
\[
\sup_{\zeta\in\domD_d} \left|
\frac{1}{\log(1+\E^{\pi\sinh\zeta})}\cdot\frac{1}{1+\E^{-\pi\sinh\zeta}}
\right|
\leq \frac{c_d}{\log(1 + c_d)},
\]
where $c_d$ is a constant defined by
\begin{equation}
 c_d = 1 + \frac{1}{\cos((\pi/2)\sin d)}.
\label{eq:c_d}
\end{equation}
\end{lemma}
\begin{proof}
Put a function $P$ as
\[
 P(z) = \frac{1}{z}\cdot(1 - \E^{-z}).
\]
By the maximum modulus principle,
$|P(\log(1+\E^{\pi\sinh \zeta}))|$
has its maximum on the boundary of $\domD_d$,
i.e., when $\zeta=x + \I d$ or $\zeta = x - \I d$.
In what follows, we consider the case $\zeta= x + \I d$,
because we can handle the case $\zeta= x - \I d$ in the same way.

Put $\xi=\log(1+\E^{\pi\sinh(x+\I d)})$
and $\gamma = -\log(\cos((\pi/2)\sin d))$.
We consider two cases: a) $|\xi|\leq \log(2+\E^{\gamma})$ and
b) $|\xi|>\log(2+\E^{\gamma})$, and for each case we show that
\[
 |P(\xi)|\leq \frac{1 + \E^{\gamma}}{\log(2 + \E^{\gamma})}=
\frac{c_d}{\log(1 + c_d)}.
\]
In the case of a), we have
\[
 |P(\xi)|
=\left|\sum_{k=1}^{\infty}\frac{(-\xi)^{k-1}}{k!}\right|
\leq\sum_{k=1}^{\infty}\frac{|\xi|^{k-1}}{k!}
=\frac{1}{|\xi|}\cdot(\E^{|\xi|} - 1)
\leq \frac{1}{\log(2 + \E^{\gamma})}(\E^{\log(2+\E^{\gamma})}-1)
=\frac{c_d}{\log(1 + c_d)},
\]
because the function $(\E^{x} - 1)/x$ is monotonically increasing.
In the case of b), from Lemma~\ref{lem:Lemma422},
it holds that
\[
 \Re\xi = \log|1 + \E^{\pi\sinh(x+\I d)}|
\geq \log[(1 + \E^{\pi\sinh(x)\cos d})\cos((\pi/2)\sin d)]
\geq -\gamma.
\]
Using this inequality, we have
\[
 |P(\xi)|\leq \frac{1}{|\xi|}(1 + |\E^{-\xi}|)
=\frac{1}{|\xi|}(1 + \E^{-\Re\xi})
\leq \frac{1}{|\xi|}(1 + \E^{\gamma}).
\]
Furthermore, because the function $1/x$ is monotonically decreasing,
we have
\[
 \frac{1}{|\xi|}(1 + \E^{\gamma})
 \leq \frac{1}{\log(2+\E^{\gamma})}(1 + \E^{\gamma})
=\frac{c_d}{\log(1 + c_d)}.
\]
This completes the proof.
\end{proof}
\begin{lemma}
\label{lem:rewrite-Maki}
Let $d$ be a positive constant with $d<\pi/2$. Then,
it holds for $z\in\phi(\domD_d)$ that
\[
 |1 - \E^{-z}|\leq \frac{c_d}{\log(1 + c_d)} |z|,
\]
where $c_d$ is a constant defined by~\eqref{eq:c_d}.
\end{lemma}
\begin{proof}
The claim immediately follows by
putting $z=\log(1 + \E^{\pi\sinh\zeta})$ in Lemma~\ref{lem:Maki}.
\end{proof}

Using these lemmas, we prove Lemma~\ref{lem:hoelder-DE} as follows.

\begin{proof}
From~\eqref{eq:Hamada-func},
using $|\tilde{f}(z) - q|\leq H |z|^{\mu}$
and $|\tilde{f}(z) - p|\leq H |\E^{-z}|^{\mu}$, we have
\[
 |f(z)| = |(\tilde{f}(z) - p)(1 - \E^{-z}) + (\tilde{f}(z) - q)\E^{-z}|
\leq H|\E^{-z}|^{\mu}|1 - \E^{-z}|  + H|z|^{\mu}|\E^{-z}|.
\]
Using Lemmas~\ref{lem:rewrite-Lemma422} and~\ref{lem:rewrite-Maki},
we have
\begin{align*}
 H|\E^{-z}|^{\mu}|1 - \E^{-z}|  + H|z|^{\mu}|\E^{-z}|
&=
 H|\E^{-z}|^{\mu}|1 - \E^{-z}|^{\mu}|1 - \E^{-z}|^{1-\mu}
  + H|z|^{\mu}|\E^{-z}|^{\mu}|\E^{-z}|^{1 - \mu}\\
&\leq H|\E^{-z}|^{\mu}|1 - \E^{-z}|^{\mu}\frac{1}{\cos^{1-\mu}((\pi/2)\sin d)}
  + H|z|^{\mu}|\E^{-z}|^{\mu}\frac{1}{\cos^{1-\mu}((\pi/2)\sin d)}\\
&\leq H|\E^{-z}|^{\mu}\left(\frac{c_d}{\log(1+c_d)}\right)^{\mu}|z|^{\mu}
\frac{1}{\cos^{1-\mu}((\pi/2)\sin d)}
  + H|z|^{\mu}|\E^{-z}|^{\mu}\frac{1}{\cos^{1-\mu}((\pi/2)\sin d)}\\
&=\frac{H}{\cos^{1-\mu}((\pi/2)\sin d)}\left\{
\left(\frac{c_d}{\log(1+c_d)}\right)^{\mu} + 1
\right\}|z|^{\mu}|\E^{-z}|^{\mu},
\end{align*}
from which the claim follows.
\end{proof}

In summary, instead of Theorem~\ref{thm:naive},
using Lemma~\ref{lem:hoelder-DE},
we establish the following theorem.

\begin{theorem}
For a given function $\tilde{f}$,
let $p$ and $q$ be defined by~\eqref{eq:def-p-and-q}.
Assume that $\tilde{f}$ is analytic in $\phi(\domD_d)$ with $0<d<\pi/2$.
Furthermore,
assume that there exists positive constants $H$ and $\mu$ with $\mu\leq 1$
such that
\begin{align*}
 |\tilde{f}(z) - q|&\leq H |z|^{\mu},\\
 |\tilde{f}(z) - p|&\leq H |\E^{-z}|^{\mu}
\end{align*}
hold for all $z\in\phi(\domD_d)$.
Let $h$ be defined as~\eqref{eq:def-h-DE}.
Then, it holds that
\begin{align*}
&\sup_{t\in(0,\infty)}\left|
\tilde{f}(t) - \left[\frac{q + p (\E^t - 1)}{\E^t}
+\sum_{j=-n}^n\left\{\tilde{f}(\phi(jh)) - \frac{q + p\E^{\pi\sinh(jh)}}{1+\E^{\pi\sinh(jh)}}\right\}
S(j,h)(\phi^{-1}(t))
\right]
\right|\\
&\leq \frac{H}{\cos^{1-\mu}((\pi/2)\sin d)}\left\{
\left(\frac{c_d}{\log(1+c_d)}\right)^{\mu} + 1
\right\} C_{\dagger} \E^{-\pi d n/\arsinh(d n/\mu)},
\end{align*}
where $C_{\dagger}$ and $c_d$ are positive constants defined
as~\eqref{eq:Cast} and~\eqref{eq:c_d}, respectively.
\end{theorem}

\section{Proof of Theorems~\ref{thm:SE-Sinc-collocation} and~\ref{thm:DE-Sinc-collocation}}
\label{sec:proofs}

In this section, we prove Theorems~\ref{thm:SE-Sinc-collocation}
and~\ref{thm:DE-Sinc-collocation}.
For both proofs,
the following lemma is useful.

\begin{lemma}[Stenger~{\cite[p.~142]{stenger93:_numer}}]
\label{lem:Stenger}
It holds for $x\in\mathbb{R}$ that
\[
 \sup_{x\in\mathbb{R}}\sum_{j=-n}^n|S(j,h)(x)|
\leq \frac{2}{\pi}\left\{\frac{3}{2} + \gamma + \log(n+1)\right\},
\]
where $\gamma$ is Euler's constant defined by
\[
 \gamma = \lim_{n\to\infty}
\left\{
1 + \frac{1}{2} + \frac{1}{3} + \cdots + \frac{1}{n-1} - \log n
\right\}
=0.5772\cdots .
\]
\end{lemma}

\subsection{Proof of Theorem~\ref{thm:SE-Sinc-collocation}}

We need the following lemma, which is similar to Lemma~\ref{lem:hoelder-DE}
in the case of the DE transformation.

\begin{lemma}
\label{lem:hoelder-SE}
For a given function $\tilde{f}$,
let $p$ and $q$ be defined by~\eqref{eq:def-p-and-q}.
Assume that $\tilde{f}$ is analytic in $\psi(\domD_d)$ with $0<d<\pi$.
Furthermore,
assume that there exists positive constants $H$, $\alpha$,
and $\beta$ with $\alpha\leq 1$ and $\beta\leq 1$
such that
\begin{align*}
 |\tilde{f}(z) - q|&\leq H \left|\frac{z}{1+z}\right|^{\alpha},\\
 |\tilde{f}(z) - p|&\leq H |\E^{-z}|^{\beta}
\end{align*}
hold for all $z\in\phi(\domD_d)$.
Then, $f$ defined by~\eqref{eq:Hamada-func}
satisfies the assumptions of Theorem~\ref{thm:SE-Sinc}.
\end{lemma}

To prove this lemma,
we prepare the following four lemmas.

\begin{lemma}[Okayama et al.~{\cite[Lemma~4.21]{Okayama-et-al}}]
\label{lem:Lemma421}
Let $x$ and $y$ be real numbers with $|y|< \pi$,
and let $\zeta=x+\I y$. Then,
\begin{align*}
 \left|\frac{1}{1+\E^{\zeta}}\right|
&\leq \frac{1}{(1+\E^{x})\cos(y/2)},\\
 \left|\frac{1}{1+\E^{-\zeta}}\right|
&\leq \frac{1}{(1+\E^{-x})\cos(y/2)}.
\end{align*}
\end{lemma}
\begin{lemma}
\label{lem:rewrite-Lemma421}
Let $d$ be a positive constant with $d<\pi$.
Then,
\begin{align*}
 \sup_{z\in\psi(\domD_d)}\left|\E^{-z}\right|
&\leq \frac{1}{\cos(d/2)},\\
 \sup_{z\in\psi(\domD_d)}\left|1 - \E^{-z}\right|
&\leq \frac{1}{\cos(d/2)}.
\end{align*}
\end{lemma}
\begin{proof}
Applying $z=\psi(\zeta)$, we have
\begin{align*}
 \sup_{z\in\psi(\domD_d)}\left|\E^{-z}\right|
&=\sup_{\zeta\in\domD_d}\left|\E^{-\psi(\zeta)}\right|
=\sup_{\zeta\in\domD_d}\left|\frac{1}{1+\E^{\zeta}}\right|,\\
 \sup_{z\in\psi(\domD_d)}\left|1 - \E^{-z}\right|
&=\sup_{\zeta\in\domD_d}\left|1 - \E^{-\psi(\zeta)}\right|
=\sup_{\zeta\in\domD_d}\left|\frac{1}{1+\E^{-\zeta}}\right|.
\end{align*}
Thus, the claim follows from Lemma~\ref{lem:Lemma421}.
\end{proof}

\begin{lemma}[Okayama and Machida~{\cite[Lemma 7]{OkayamaMachida}}]
\label{lem:Machida}
Let $d$ be a positive constant with $d<\pi$. Then,
\[
\sup_{\zeta\in\domD_d} \left|
\frac{1+\log(1+\E^{\zeta})}{\log(1+\E^{\zeta})}\cdot\frac{1}{1+\E^{-\zeta}}
\right|
\leq \frac{1 + \log(1 + \tilde{c}_d)}{\log(1 + \tilde{c}_d)}\tilde{c}_d,
\]
where $\tilde{c}_d$ is a constant defined by
\begin{equation}
 \tilde{c}_d = 1 + \frac{1}{\cos(d/2)}.
\label{eq:tilde_c_d}
\end{equation}
\end{lemma}
\begin{lemma}
\label{lem:rewrite-Machida}
Let $d$ be a positive constant with $d<\pi$. Then,
it holds for $z\in\psi(\domD_d)$ that
\[
 |1 - \E^{-z}|\leq \frac{1 + \log(1 + \tilde{c}_d)}{\log(1 + \tilde{c}_d)}\tilde{c}_d \left|\frac{z}{1+z}\right|,
\]
where $\tilde{c}_d$ is a constant defined by~\eqref{eq:tilde_c_d}.
\end{lemma}
\begin{proof}
The claim immediately follows by
putting $z=\log(1 + \E^{\zeta})$ in Lemma~\ref{lem:Machida}.
\end{proof}

Using these lemmas, we prove Lemma~\ref{lem:hoelder-SE} as follows.

\begin{proof}
From~\eqref{eq:Hamada-func},
using $|\tilde{f}(z) - q|\leq H |z/(1+z)|^{\alpha}$
and $|\tilde{f}(z) - p|\leq H |\E^{-z}|^{\beta}$, we have
\[
 |f(z)| = |(\tilde{f}(z) - p)(1 - \E^{-z}) + (\tilde{f}(z) - q)\E^{-z}|
\leq H|\E^{-z}|^{\beta}|1 - \E^{-z}|
 + H\left|\frac{z}{1+z}\right|^{\alpha}|\E^{-z}|.
\]
Using Lemmas~\ref{lem:rewrite-Lemma421} and~\ref{lem:rewrite-Machida},
we have
\begin{align*}
 H|\E^{-z}|^{\beta}|1 - \E^{-z}|
 + H\left|\frac{z}{1+z}\right|^{\alpha}|\E^{-z}|
&=
 H|\E^{-z}|^{\beta}|1 - \E^{-z}|^{\alpha}|1 - \E^{-z}|^{1-\alpha}
  + H\left|\frac{z}{1+z}\right|^{\alpha}|\E^{-z}|^{\beta}|\E^{-z}|^{1 - \beta}\\
&\leq H|\E^{-z}|^{\beta}|1 - \E^{-z}|^{\alpha}\frac{1}{\cos^{1-\alpha}(d/2)}
  + H\left|\frac{z}{1+z}\right|^{\alpha}
|\E^{-z}|^{\beta}\frac{1}{\cos^{1-\beta}(d/2)}\\
&\leq H|\E^{-z}|^{\beta}
 \left(\frac{1 + \log(1 + \tilde{c}_d)}{\log(1 + \tilde{c}_d)}\tilde{c}_d
 \left|\frac{z}{1+z}\right|\right)^{\alpha}\frac{1}{\cos^{1-\alpha}(d/2)}\\
&\quad
  + H\left|\frac{z}{1+z}\right|^{\alpha}
|\E^{-z}|^{\beta}\frac{1}{\cos^{1-\beta}(d/2)}\\
&=H\left\{
\left(
\frac{1 + \log(1 + \tilde{c}_d)}{\log(1 + \tilde{c}_d)}\tilde{c}_d
\right)^{\alpha}\frac{1}{\cos^{1-\alpha}(d/2)}
+\frac{1}{\cos^{1-\beta}(d/2)}
\right\}\left|\frac{z}{1+z}\right|^{\alpha}|\E^{-z}|^{\beta},
\end{align*}
from which the claim follows.
\end{proof}

Thanks to Lemma~\ref{lem:hoelder-SE},
we can apply Theorem~\ref{thm:SE-Sinc-boundary}
to the solution $y_i$ under the assumptions of
Theorem~\ref{thm:SE-Sinc-collocation}.
Based on this idea, we prove
Theorem~\ref{thm:SE-Sinc-collocation} as follows.

\begin{proof}
Let $\mathbd{p}=\lim_{t\to\infty}\mathbd{y}(t)$,
and let $\hat{\mathbd{y}}$ be a function defined by
\[
 \hat{\mathbd{y}}(t)
=\frac{\mathbd{r} + \mathbd{p}(\E^{t} - 1)}{\E^{t}}
+\sum_{j=-M}^N\left\{
\mathbd{y}(\phi(jh)) - \frac{\mathbd{r}+\mathbd{p}\E^{jh}}{1+\E^{jh}}
S(j,h)(\psi^{-1}(t))
\right\},
\]
which is derived by application of~\eqref{eq:general-SE-Sinc}
to the solution $\mathbd{y}(t)$.
Note that the approximate solution $\hat{\mathbd{y}}^{(l)}(t)$
is defined by~\eqref{eq:ynSE-SCM}, and
$\mathbd{p}^{(\psi)}$ satisfies
$\mathbd{p}^{(\psi)} = \lim_{t\to\infty}\mathbd{y}^{(l)}(t)$.
To estimate the error of $\hat{y}_i^{(l)}(t)$,
we insert $\hat{y}_i(t)$ as
\begin{equation}
 |y_i(t) - \hat{y}_i^{(l)}(t)|
\leq |y_i(t) - \hat{y}_i(t)|
+|\hat{y}_i(t) - \hat{y}_i^{(l)}(t)|.
\label{eq:hat_y_i}
\end{equation}
As for the first term,
according to Theorem~\ref{thm:SE-Sinc-boundary},
there exists a positive constant $\tilde{C}_i$ independent of $n$
such that
\begin{equation}
 |y_i(t) - \hat{y}_i(t)|\leq \tilde{C}_i\sqrt{n}\E^{-\sqrt{\pi d \mu n}}
\leq \frac{\tilde{C}_i}{\log 2}\log(n+1)\sqrt{n}\E^{-\sqrt{\pi d \mu n}}
\leq \frac{\tilde{C}}{\log 2}\log(n+1)\sqrt{n}\E^{-\sqrt{\pi d \mu n}},
\label{leq:first-term}
\end{equation}
where $\tilde{C}=\max_{i=1,\,\ldots,\,m}\tilde{C}_i$.
Next, we estimate the second term.
First, it is rewritten as
\begin{align*}
 |\hat{y}_i(t) - \hat{y}^{(l)}_i(t)|
&=\left|
\frac{(p_i - p_i^{(\psi)})(\E^{t} - 1)}{\E^{t}}
+\sum_{j=-M}^N
\left\{y_i(\psi(jh)) - y_i^{(l)}(\psi(jh))\right\}S(j,h)(\psi^{-1}(t))
-\sum_{j=-M}^N
\frac{(p_i - p_i^{(\psi)})\E^{jh}}{1+\E^{jh}}S(j,h)(\psi^{-1}(t))
\right|.
\end{align*}
Noting
\[
 p_i - p^{(\psi)}_i = \lim_{t\to\infty}
\left(y_i(t) - y^{(l)}_i(t)\right),
\]
according to Theorem~\ref{thm:SE-Sinc-Nystroem},
there exist positive constants $C$ and $\hat{C}$ such that
\begin{align*}
|p_i - p_i^{(\psi)}|
=\lim_{t\to\infty}
\left|y_i(t) - y^{(l)}_i(t)\right|
&\leq \sup_{t\in(0,\infty)}
\left|y_i(t) - y^{(l)}_i(t)\right|
\leq \left(C + \hat{C}\|A_{lm}^{-1}\|_{\infty}\right)
\sqrt{n}\E^{-\sqrt{\pi d \mu n}},\\
\left|y_i(\psi(jh)) - y_i^{(l)}(\psi(jh))\right|
&\leq\sup_{t\in(0,\infty)}
\left|y_i(t) - y^{(l)}_i(t)\right|
\leq \left(C + \hat{C}\|A_{lm}^{-1}\|_{\infty}\right)
\sqrt{n}\E^{-\sqrt{\pi d \mu n}}.
\end{align*}
Therefore, we have
\begin{align*}
&\left|
\frac{(p_i - p_i^{(\psi)})(\E^{t} - 1)}{\E^{t}}
+\sum_{j=-M}^N
\left\{y_i(\psi(jh)) - y_i^{(l)}(\psi(jh))\right\}S(j,h)(\psi^{-1}(t))
-\sum_{j=-M}^N
\frac{(p_i - p_i^{(\psi)})\E^{jh}}{1+\E^{jh}}S(j,h)(\psi^{-1}(t))
\right|\\
&\leq
\left(\left|\frac{\E^t - 1}{\E^t}\right|
+\sum_{j=-M}^N|S(j,h)(\psi^{-1}(t))|
+\sum_{j=-M}^N\frac{\E^{jh}}{1+\E^{jh}}|S(j,h)(\psi^{-1}(t))|
\right)\left(C + \hat{C}\|A_{lm}^{-1}\|_{\infty}\right)
\sqrt{n}\E^{-\sqrt{\pi d \mu n}}\\
&\leq
\left(1
+\sum_{j=-M}^N|S(j,h)(\psi^{-1}(t))|
+\sum_{j=-M}^N|S(j,h)(\psi^{-1}(t))|
\right)\left(C + \hat{C}\|A_{lm}^{-1}\|_{\infty}\right)
\sqrt{n}\E^{-\sqrt{\pi d \mu n}}\\
&\leq\left(1
+2\sum_{j=-n}^n|S(j,h)(\psi^{-1}(t))|
\right)\left(C + \hat{C}\|A_{lm}^{-1}\|_{\infty}\right)
\sqrt{n}\E^{-\sqrt{\pi d \mu n}},
\end{align*}
where $n=\max\{M, N\}$ is used in the last inequality,
because of~\eqref{eq:def-MN}.
Furthermore, using Lemma~\ref{lem:Stenger},
we have the final estimate for the second term as
\begin{align}
  |\hat{y}_i(t) - \hat{y}^{(l)}_i(t)|
&\leq \left(
1 + \frac{6+4\gamma + 4\log(n+1)}{\pi}
\right)\left(C + \hat{C}\|A_{lm}^{-1}\|_{\infty}\right)
\sqrt{n}\E^{-\sqrt{\pi d \mu n}}\nonumber\\
&\leq \left(
\frac{\pi + 6+4\gamma}{\pi\log 2} + \frac{4}{\pi}
\right)\left(C + \hat{C}\|A_{lm}^{-1}\|_{\infty}\right)
\log(n+1)
\sqrt{n}\E^{-\sqrt{\pi d \mu n}}.
\label{leq:second-term}
\end{align}
Combining the estimates~\eqref{leq:first-term} and~\eqref{leq:second-term},
we obtain the claim.
\end{proof}

\subsection{Proof of Theorem~\ref{thm:DE-Sinc-collocation}}

Thanks to Lemma~\ref{lem:hoelder-DE},
we can apply Theorem~\ref{thm:DE-Sinc-boundary}
to the solution $y_i$ under the assumptions of
Theorem~\ref{thm:DE-Sinc-collocation}.
Based on this idea, we prove
Theorem~\ref{thm:DE-Sinc-collocation} as follows.

\begin{proof}
Let $\mathbd{p}=\lim_{t\to\infty}\mathbd{y}(t)$,
and let $\hat{\mathbd{y}}$ be a function defined by
\[
 \hat{\mathbd{y}}(t)
=\frac{\mathbd{r} + \mathbd{p}(\E^{t} - 1)}{\E^{t}}
+\sum_{j=-n}^n\left\{
\mathbd{y}(\phi(jh)) - \frac{\mathbd{r}+\mathbd{p}\E^{\pi\sinh(jh)}}{1+\E^{\pi\sinh(jh)}}
S(j,h)(\phi^{-1}(t))
\right\},
\]
which is derived by application of~\eqref{eq:general-DE-Sinc}
to the solution $\mathbd{y}(t)$.
Note that the approximate solution $\hat{\mathbd{y}}^{(l)}(t)$
is defined by~\eqref{eq:ynDE-SCM}, and
$\mathbd{p}^{(\phi)}$ satisfies
$\mathbd{p}^{(\phi)} = \lim_{t\to\infty}\mathbd{y}^{(l)}(t)$.
To estimate the error of $\hat{y}_i^{(l)}(t)$,
we insert $\hat{y}_i(t)$ as~\eqref{eq:hat_y_i}.
As for the first term of~\eqref{eq:hat_y_i},
according to Theorem~\ref{thm:DE-Sinc-boundary},
there exists a positive constant $\tilde{C}_i$ independent of $n$
such that
\begin{align}
 |y_i(t) - \hat{y}_i(t)|\leq \tilde{C}_i \E^{-\pi d n/\arsinh(dn/\mu)}
&\leq \frac{\tilde{C}_i}{\arsinh(d/\mu)\log 2}\arsinh(dn/\mu)\log(n+1)\E^{-\pi d n/\arsinh(dn/\mu)}\nonumber\\
&\leq \frac{\tilde{C}}{\arsinh(d/\mu)\log 2}\arsinh(dn/\mu)\log(n+1)\E^{-\pi d n/\arsinh(dn/\mu)},
\label{leq:first-term-DE}
\end{align}
where $\tilde{C}=\max_{i=1,\,\ldots,\,m}\tilde{C}_i$.
Next, we estimate the second term of~\eqref{eq:hat_y_i}.
First, it is rewritten as
\begin{align*}
& |\hat{y}_i(t) - \hat{y}^{(l)}_i(t)|\\
&=\left|
\frac{(p_i - p_i^{(\phi)})(\E^{t} - 1)}{\E^{t}}
+\sum_{j=-n}^n
\left\{y_i(\phi(jh)) - y_i^{(l)}(\phi(jh))\right\}S(j,h)(\phi^{-1}(t))
-\sum_{j=-n}^n
\frac{(p_i - p_i^{(\phi)})\E^{\pi\sinh(jh)}}{1+\E^{\pi\sinh(jh)}}S(j,h)(\phi^{-1}(t))
\right|.
\end{align*}
Noting
\[
 p_i - p^{(\phi)}_i = \lim_{t\to\infty}
\left(y_i(t) - y^{(l)}_i(t)\right),
\]
according to Theorem~\ref{thm:DE-Sinc-Nystroem},
there exist positive constants $C$ and $\hat{C}$ such that
\begin{align*}
|p_i - p_i^{(\phi)}|
=\lim_{t\to\infty}
\left|y_i(t) - y^{(l)}_i(t)\right|
&\leq \sup_{t\in(0,\infty)}
\left|y_i(t) - y^{(l)}_i(t)\right|
\leq \left(C + \hat{C}\|B_{lm}^{-1}\|_{\infty}\right)
\arsinh(dn/\mu)\E^{-\pi d n/\arsinh(dn/\mu)},\\
\left|y_i(\phi(jh)) - y_i^{(l)}(\phi(jh))\right|
&\leq\sup_{t\in(0,\infty)}
\left|y_i(t) - y^{(l)}_i(t)\right|
\leq \left(C + \hat{C}\|B_{lm}^{-1}\|_{\infty}\right)
\arsinh(dn/\mu)\E^{-\pi d n/\arsinh(dn/\mu)}.
\end{align*}
Therefore, we have
\begin{align*}
&\left|
\frac{(p_i - p_i^{(\phi)})(\E^{t} - 1)}{\E^{t}}
+\sum_{j=-n}^n
\left\{y_i(\phi(jh)) - y_i^{(l)}(\phi(jh))\right\}S(j,h)(\phi^{-1}(t))
-\sum_{j=-n}^n
\frac{(p_i - p_i^{(\phi)})\E^{\pi\sinh(jh)}}{1+\E^{\pi\sinh(jh)}}S(j,h)(\phi^{-1}(t))
\right|\\
&\leq
\left(\left|\frac{\E^t - 1}{\E^t}\right|
+\sum_{j=-n}^n|S(j,h)(\phi^{-1}(t))|
+\sum_{j=-n}^n\frac{\E^{\pi\sinh(jh)}}{1+\E^{\pi\sinh(jh)}}|S(j,h)(\phi^{-1}(t))|
\right)\left(C + \hat{C}\|B_{lm}^{-1}\|_{\infty}\right)
\arsinh(dn/\mu)\E^{-\pi d n/\arsinh(dn/\mu)}\\
&\leq
\left(1
+\sum_{j=-n}^n|S(j,h)(\phi^{-1}(t))|
+\sum_{j=-n}^n|S(j,h)(\phi^{-1}(t))|
\right)\left(C + \hat{C}\|B_{lm}^{-1}\|_{\infty}\right)
\arsinh(dn/\mu)\E^{-\pi d n/\arsinh(dn/\mu)}\\
&=\left(1
+2\sum_{j=-n}^n|S(j,h)(\phi^{-1}(t))|
\right)\left(C + \hat{C}\|B_{lm}^{-1}\|_{\infty}\right)
\arsinh(dn/\mu)\E^{-\pi d n/\arsinh(dn/\mu)}.
\end{align*}
Furthermore, using Lemma~\ref{lem:Stenger},
we have the final estimate for the second term as
\begin{align}
  |\hat{y}_i(t) - \hat{y}^{(l)}_i(t)|
&\leq \left(
1 + \frac{6+4\gamma + 4\log(n+1)}{\pi}
\right)\left(C + \hat{C}\|B_{lm}^{-1}\|_{\infty}\right)
\arsinh(dn/\mu)\E^{-\pi d n/\arsinh(dn/\mu)}\nonumber\\
&\leq \left(
\frac{\pi + 6+4\gamma}{\pi\log 2} + \frac{4}{\pi}
\right)\left(C + \hat{C}\|B_{lm}^{-1}\|_{\infty}\right)
\log(n+1)
\arsinh(dn/\mu)\E^{-\pi d n/\arsinh(dn/\mu)}.
\label{leq:second-term-DE}
\end{align}
Combining the estimates~\eqref{leq:first-term-DE}
and~\eqref{leq:second-term-DE},
we obtain the claim.
\end{proof}


\bibliography{sinc-coll-initval}

\end{document}